\documentclass[a4paper,10pt,reqno]{amsart}
\usepackage{hyperref,amssymb,a4wide}

\newtheorem{engtheorem}{Theorem}
\newtheorem{engcorollary}[engtheorem]{Corollary}
\newtheorem{englemma}[engtheorem]{Lemma}
\newtheorem{engproposition}[engtheorem]{Proposition}
\newtheorem{engproblem}[engtheorem]{Problem}

\theoremstyle{definition} 
\newtheorem{engdefinition}[engtheorem]{Definition}
\newtheorem{engexample}[engtheorem]{Example}

\newcommand{\IN}{\mathbb N}
\newcommand{\IR}{\mathbb R}
\newcommand{\IQ}{\mathbb Q}

\newcommand{\w}{\omega}

\newcommand{\Ra}{\Rightarrow}
\newcommand{\cl}{\mathrm{cl}}

\usepackage[all]{xy}

\title[On some functional generalizations of the regularity]{ON SOME FUNCTIONAL GENERALIZATIONS OF THE REGULARITY OF TOPOLOGICAL SPACES}
\dedicatory{Dedicated to the 60th birthday of M. M. Zarichnyi}
\author{Taras BANAKH, Bogdan BOKALO}
\address{Ivan Franko National University of Lviv, 
Universytetska Str.,  1, 79000, Lviv, Ukraine}
\email{t.o.banakh@gmail.com, b.m.bokalo.gmail.com}

\begin{document}

\begin{abstract}
We introduce and study some generalizations of regular spaces, which were motivated by studying continuity properties of functions between (regular) topological spaces. In particular, we prove that a first-countable Hausdorff topological space is regular if and only if it does not contain a topological copy of the Gutik hedgehog.\end{abstract}
\maketitle

In this paper we introduce and study some generalizations of regular spaces, which were motivated by continuity properties of functions between (regular) topological spaces. First we introduce the necessary definitions.
\smallskip

A subset $U$ of a topological space $X$ is called {\em $\theta$-open} if each point $x\in U$ has a neighborhood $O_x\subset X$ such that $\bar O_x\subset U$. It is clear that each $\theta$-open set is open. Moreover, a topological space is {\em regular} if and only if each open subset of $X$ is $\theta$-open.

\begin{englemma}\label{l:t-open} Let $U$ be a $\theta$-open subset of a topological space $X$ and $V$ be a $\theta$-open subset of $U$. Then $V$ is $\theta$-open in $X$.
\end{englemma}

\begin{proof}[\textsl{Proof}]  For each point $x\in V$, the $\theta$-openness of $U$ in $X$ yields an open neighborhood $U_x\subset X$ such that $\cl_X(U_x)\subset U$. The $\theta$-openness of $V$ in $U$ yields an open neughborhood $V_x\subset U$ such that $\cl_U(V_x)\subset U$. Now consider the open neighborhood $O_x=V_x\cap U_x$ and observe that $\cl_X(O_x)\subset \cl_X(V_x)\cap\cl_X(U_x)\subset \cl_X(V_x)\cap U=\cl_U(V_x)\subset V$.
\end{proof}

For a function $f:X\to Y$ between topological spaces by $C(f)$ we denote the set of continuity points of $f$.

\begin{engdefinition} A function $f:X\to Y$ beween topological spaces is called
\begin{itemize}
\item {\em scatteredly continuous} if for any non-empty subset $A\subset X$ the set $C(f|A)$ is not empty;
\item {\em weakly discontinuous} if if for any non-empty subset $A\subset X$ the set $C(f|A)$ has non-empty interior in $A$;
\item {\em $\theta$-weakly discontinuous} if if for any non-empty subset $A\subset X$ the set $C(f|A)$ contains a non-empty $\theta$-open subset of $A$.
\end{itemize}
\end{engdefinition}

So, we have the implications:
$$\mbox{$\theta$-weakly discontinuous $\Ra$ weakly discontinuous $\Ra$ scatteredly continuous}.$$

The first and last implications can be reversed for functions with regular domain and range, respectively.

\begin{engtheorem}[trivial]\label{t:trivial}  A function $f:X\to Y$ from a regular topological space $X$ to a topological space $Y$  is weakly discontinuous if and only if it is $\theta$-weakly discontinuous.
\end{engtheorem}

\begin{engtheorem}[Bokalo]\label{t:Bokalo} A function $f:X\to Y$ from a topological space $X$ to a regular space $Y$   is scatteredly continuous if and only if it is weakly discontinuous.
\end{engtheorem}

A proof the Theorem~\ref{t:Bokalo} can be found in \cite{AB}, \cite{BM}. More information on various sorts of generalized continuity can be found in \cite{Ba}--\cite{Vino}.

Motivated by Theorems~\ref{t:trivial} and \ref{t:Bokalo}, let us introduce the following definition.

\begin{engdefinition} A topological space $X$ is called
\begin{itemize}
\item  {\em $sw$-regular} if any scatteredly continuous function $f:Z\to X$ defined on a topological space $Z$ is weakly discontinuous;
\item  {\em $w\theta$-regular} if any weakly discontinuous function $f:X\to Y$ to any topological space $Y$ is $\theta$-weakly discontinuous.
\end{itemize}
\end{engdefinition}

Theorems~\ref{t:trivial} and \ref{t:Bokalo} imply that each regular space is $sw$-regular and $w\theta$-regular.

The following theorem characterizes $w\theta$-regular spaces.

\begin{engtheorem}\label{t:wt-char} A topological space $X$ is $w\theta$-regular if and only if for each subspace $A\subset X$, each non-empty open subset $U\subset A$ contains a non-empty $\theta$-open subset of $A$.
\end{engtheorem}

\begin{proof}[\textsl{Proof}] To prove the ``if'' part, assume that  for each  subspace $A\subset X$, every non-empty open subset $U\subset A$ contains a non-empty $\theta$-open subset of $A$. To show that the space $X$ is $w\theta$-regular, fix any weakly discontinuous map $f:X\to Y$. To show that $f$ is $\theta$-weakly discontinuous, take any non-empty subset $A\subset X$. Since $f$ is weakly discontinuous, there exists a non-empty open subset $U\subset A$ such that $f|U$ is continuous. By our assumption, $U$ contains a $\theta$-open subspace $V$ of $A$.  Since $f|V$ is continuous, the function $f$ is $\theta$-weakly discontinuous.

Now we prove the ``only if'' part. Assume that the space $X$ is $w\theta$-regular. Given any subset $A\subset X$ and a non-empty open subset $U\subset A$, consider the
closures $\bar A$ and $\overline{A\setminus U}$ of the sets $A$ and $A\setminus U$ in $X$. Observe that $\tilde U:=\bar A\setminus \overline{A\setminus U}$ is an open set in $\bar A$ with $\tilde U\cap A=U$ and $\tilde U\subset \overline{U}$. Consider the topological sum $Y=\tilde U\oplus (X\setminus \tilde U)$ and observe that the identity map $f:X\to Y$ is weakly discontinuous. The $w\theta$-regularity of the space $X$ ensures that $f$ is $\theta$-weakly discontinuous. Consequently, the closure $\bar U$ of $U$ in $\bar A$ contains a non-empty $\theta$-open subset $V\subset\bar U$ such that $f|V$ is continuous. The continuity of $f|V$ ensures that $V\subset \tilde U$. We claim that $V$ is $\theta$-open in $\bar A$. Since $V$ is $\theta$-open in $\bar U$, for any $x\in V$ there exists a neighborhood $O_x$ of $x$ such that $O_x$ is open in $\bar U$ and $O_x\subset \overline{O}_x\subset V\subset\tilde U$. So, $O_x$ is open in $\tilde U$ and hence is open in $\bar A$.

Taking into account that $V$ is a non-empty $\theta$-open subset of $\bar A$, we conclude that $V\cap A\subset \tilde U\cap A=U$ is a non-empty $\theta$-open subset of $A$, contained in the set $U$.
\end{proof}

\begin{engproblem} Characterize topological spaces, which are $sw$-regular.
\end{engproblem}

We shall prove that $sw$-regular and $w\theta$-regular spaces are preserved by $\theta$-weak homeomorphisms.

\begin{engdefinition} A bijective function $f:X\to Y$ between topological spaces is called a ($\theta$-){\em weak homeomorphism} if both functions $f$ and $f^{-1}$ are ($\theta$-)weakly discontinuous.
\end{engdefinition}

We shall need the following proposition describing the continuity properties of compositions of scatteredly continuous, weakly discontinuous and $\theta$-weakly discontinuous  functions.

\begin{engproposition}\label{p:composition} Let $f:X\to Y$ and $g:Y\to Z$ be two functions between topological spaces.
\begin{enumerate}
\item If $f,g$ are weakly discontinuous, then  $g\circ f$ is weakly discontinuous.
\item If $f,g$ are $\theta$-weakly discontinuous, then  $g\circ f$ is $\theta$-weakly discontinuous.
\item If $f$ is weakly discontinuous and $g$ is scatteredly continuous, then $g\circ f$ is scatteredly continuous.
\item If $f$ is scatteredly continuous and $g$ is $\theta$-weakly discontinuous, then $g\circ f$ is scatteredly continuous.
\end{enumerate}
\end{engproposition}

\begin{proof}[\textsl{Proof}] 1. Assume that $f,g$ are weakly discontinuous. To prove that $g\circ f$ is weakly discontinuous, we need to show that for any non-empty subset $A\subset X$ the set $C(g\circ f|A)$ has non-empty interior in $A$. By the weak discontinuity of $f$, the set $C(f|A)$ contains a non-empty open subset $U\subset A$. By the weak discontinuity of $g$, the set $C(g|f(U))$ contains a non-empty open set $V\subset f(U)$. By the continuity of $f|U$, the set $W=(f|U)^{-1}(V)$ is open in $U$ and hence open in $A$. Since $f(W)\subset V$, the continuity of the restrictions $f|W$ and $g|V$ implies the continuity of the restriction $g\circ f|W$. So, $W\subset C(g\circ f|A)$.
\smallskip

2. Assume that $f,g$ are $\theta$-weakly discontinuous. To prove that $g\circ f$ is $\theta$-weakly discontinuous, we need to show that for any non-empty subset $A\subset X$ the set $C(g\circ f|A)$ contains a non-empty $\theta$-open subset $W\subset A$. By the $\theta$-weak discontinuity of $f$, the set $C(f|A)$ contains a non-empty $\theta$-open subset $U\subset A$. By the $\theta$-weak discontinuity of $g$, the set $C(g|f(U))$ contains a non-empty $\theta$-open set $V\subset f(U)$. By the continuity of $f|U$, the set $W=(f|U)^{-1}(V)$ is $\theta$-open in $U$ and hence $\theta$-open in $A$, by Lemma~\ref{l:t-open}. Since $f(W)\subset V$, the continuity of the restrictions $f|W$ and $g|V$ implies the continuity of the restriction $g\circ f|W$. Now we see that the set $C(g\circ f|A)$ contains the non-empty $\theta$-open subset $W$ of $A$, witnessing that $g\circ f$ is $\theta$-weakly discontinuous.
\smallskip

3. Assume that $f$ is weakly discontinuous and $g$ is scatteredly continuous. To prove that $g\circ f$ is scatteredly continuous, we need to show that for any non-empty subset $A\subset X$ the function $g\circ f|A$ has a continuity point. By the weak discontinuity of $f$, the set $C(f|A)$ contains a non-empty open subset $U\subset A$. By the scattered continuity of $g$, the function $g|f(U)$ has a continuity point $y$. Then any point $x\in U\cap f^{-1}(y)$ is a continuity point of the restriction $g\circ f|A$.
\smallskip

4. Assume that $f$ is scatteredly continuous and $g$ is $\theta$-weakly discontinuous. Given a non-empty subset $A\subset X$, we need to show that the restriction $g\circ f|A$ has a continuity point. Let $A_0:=A$ and $A_\alpha:=\bigcap_{\beta<\alpha}A_\beta\setminus C(f|A_\beta)$ for any non-zero ordinal $\alpha$. In particular, $A_{\alpha+1}=A_\alpha\setminus C(f|A_\alpha)$ for any ordinal $\alpha$.

Let $\delta$ be the smallest ordinal such that $A_\delta$ is not dense in $A$ and let $W=A\setminus\overline{A}_\delta$. It follows that $W=\bigcup_{\alpha<\delta}W\cap C(f|A_\alpha)$ and each set $W\cap C(f|A_\alpha)$ is dense in $W$ (by the scattered continuity of $f$).

Since the function $g$ is $\theta$-weakly discontinuous, the set $C(g|f(W))$ contains a non-empty $\theta$-open subset $V\subset f(W)$. Since $W=\bigcup_{\alpha<\delta}W\cap C(f|A_\alpha)$, we can choose the smallest ordinal $\gamma<\delta$ such that $W\cap C(f|A_\gamma)\cap f^{-1}(V)\ne\emptyset$. Choose a point $x\in W\cap C(f|A_\gamma)\cap f^{-1}(V)$.
Since the set $V$ is $\theta$-open in $f(W)$, the point $f(x)\in V$ has a closed neighborhood $\bar O_{f(x)}\subset f(W)$ such that $\bar O_{f(x)}\subset V$. By the continuity of the map $f|A_\gamma$ at $x$, there exists an open neighborhood $O_x\subset W$ of $x$ such that $f(O_x\cap A_\gamma)\subset \bar O_{f(x)}\subset V$.

We claim that $\gamma=0$. To derive a contradiction, assume that $\gamma>0$. In this case $W\cap C(f|A_0)\cap f^{-1}(V)=\emptyset$ and hence $x\notin C(f|A_0)=C(f|A)$. By the density of $C(f|A)$ in $A$, there exists a point $z\in O_x\cap C(f|A)$. It follows that $f(z)\in W\setminus V\subset W\setminus \bar O_{f(x)}$. By the continuity of $f|W$ at $z$, there exists an open neighborhood $O_z\subset O_x$ such that $f(O_z)\subset f(W)\setminus \bar O_{f(x)}$. Then $$f(O_z\cap A_\gamma)=f(O_z\cap O_x\cap A_\gamma)\subset f(O_z)\cap f(O_x\cap A_\gamma)\subset (f(W)\setminus\bar O_{f(x)})\cap\bar O_{f(x)}=\emptyset$$ and hence $O_z\cap A_\gamma=\emptyset$, which contradicts the density of $A_\gamma$ in $A$.
This contradiction shows that $\gamma=0$ and hence $x\in C(f|A_\gamma)=C(f|A)$ is a continuity point of $f|A$ with $f(O_x)\subset V$. The continuity of the restriction $g|V$ implies that $g\circ f|A$ is continuous at $x$. So, $g\circ f|A$ has a continuity point.
\end{proof}

\begin{engtheorem} A topological space $X$ is $sw$-regular if there exists a $\theta$-weakly discontinuous bijective function $h:X\to Y$ to an $sw$-regular space $Y$ such that $h^{-1}$ is weakly discontinuous.
\end{engtheorem}

\begin{proof}[\textsl{Proof}] 
To show that $X$ is $sw$-regular, we need to show that each scatteredly continuous function $f:Z\to X$ is weakly discontinuous. By Proposition~\ref{p:composition}(4), the composition $h\circ f:Z\to Y$ is scatteredly continuous. Since $Y$ is $sw$-regular, the function $h\circ f$ is weakly discontinuous. By Proposition~\ref{p:composition}(1), the  composition $h^{-1}\circ h\circ f=f$ is weakly discontinuous.
\end{proof}

\begin{engtheorem} 
A topological space $X$ is $w\theta$-regular if there exists a $\theta$-weakly discontinuous bijective function $h:X\to Y$ to a $w\theta$-regular space $Y$ such that $h^{-1}$ is weakly discontinuous.
\end{engtheorem}

\begin{proof}[\textsl{Proof}] 
To see that $X$ is $w\theta$-regular, we need to show that each weakly discontinuous function $f:X\to Z$ is $\theta$-weakly discontinuous. By Proposition~\ref{p:composition}(1), the composition $f\circ h^{-1}:Y\to Z$ is weakly discontinuous. Since $Y$ is $w\theta$-regular, the function $f\circ h^{-1}$ is $\theta$-weakly discontinuous. By Proposition~\ref{p:composition}(2), the  composition $f\circ h^{-1}\circ h=f$ is $\theta$-weakly discontinuous.
\end{proof}

\begin{engcorollary}\label{c:homo} 
The classes of $sw$-regular and $w\theta$-regular spaces are preserved by $\theta$-weak homeomorphisms.
\end{engcorollary}

\begin{engdefinition} 
A topological space $X$ is called ($\theta$-){\em weakly regular} if it is ($\theta$-)weakly homeomorphic to a regular topological space.
\end{engdefinition}

\begin{engexample}\label{ex:Bokalo} 
Consider the real line $\IR$ endowed with the second-countable topology $\tau$ generated by the subbase $$\{\IQ\}\cup\{(-\infty,a),(a,+\infty):a\in\IR\}.$$
It can be shown that the topological space $X=(\IR,\tau)$ is weakly regular. The identity map $\IR\to X$ is scatteredly continuous but not weakly discontinuous, which implies that the space $X$ is not $sw$-regular. On the other hand, the function $\chi:X\to\{0,1\}\subset\IR$ defined by $$\chi(x)=\begin{cases}
1&\mbox{if $x\in\IQ$};\\
0&\mbox{if $x\in\IR\setminus\IQ$;}
\end{cases}
$$is weakly discontinuous but not $\theta$-weakly discontinuous, witnessing that the space $X$ is not $w\theta$-regular. Theorem~\ref{t:wr=>sw+twr} implies that the space $X$ is not $\theta$-weakly regular.
\end{engexample}

Theorem~\ref{t:trivial}, \ref{t:Bokalo} and Corollary~\ref{c:homo} imply:

\begin{engtheorem}\label{t:wr=>sw+twr} 
Each $\theta$-weakly regular space is $sw$-regular and $w\theta$-regular.
\end{engtheorem}

\begin{engtheorem}\label{t:twr} 
A topological space $X$ is $\theta$-weakly regular if and only if each non-empty (closed) subspace $A\subset X$ contains a non-empty $\theta$-open regular subspace.
\end{engtheorem}

\begin{proof}[\textsl{Proof}] 
First assume that $X$ is $\theta$-weakly regular and fix any $\theta$-weak homeomorphism $h:X\to Y$ to a regular topological space $Y$.

Given any subspace $A\subset X$, we need to find a non-empty $\theta$-open regular subspace $W\subset A$. Since the map $h$ is $\theta$-weakly discontinuous, there exists a non-empty $\theta$-open subset $U\subset A$ such that $h|U$ is continuous. Since $h^{-1}$ is $\theta$-weakly discontinuous, the non-empty subspace  $h(U)$ of $Y$ contains a non-empty $\theta$-open subspace $V$ such that $h^{-1}|V$ is continuous. The continuity of the map $h|U$ implies that the set $W:=(h|U)^{-1}(V)$ is $\theta$-open in $U$ and hence $\theta$-open in $A$ (by Lemma~\ref{l:t-open}). The continuity of maps $h|W$ and $h^{-1}|h(W)$ implies that $h|W:W\to h(W)$ is a homeomorphism. The regularity of the topological space $Y$ implies the regularity of its subspace $h(W)$ and the regularity of the topological copy $W$ of $h(W)$. Therefore, $W$ is a required non-empty $\theta$-open regular subspace of $A$.
\smallskip

Now assume that each non-empty closed  subspace $A\subset X$ contains a non-empty $\theta$-open regular subspace. Let $A^\theta$ be the union of all $\theta$-open regular subspaces of $A$. It is clear that the subspace $A^\theta$ is $\theta$-open in $A$ and regular.  Let $X_0:=X$ and $X_\alpha=\bigcap_{\beta<\alpha}X_\beta\setminus X_\beta^\theta$ for each ordinal $\alpha$. It follows that for any ordinal $\alpha$ with $X_\alpha\ne\emptyset$ the set $X_{\alpha+1}=X_\alpha\setminus X_\alpha^\theta$ is closed in $X_\alpha$ and has non-empty complement $X_{\alpha+1}\setminus X_\alpha=X_\alpha^\theta$. Consequently, $X_\gamma=\emptyset$ for some $\gamma$ and hence $X=\bigcup_{\alpha<\gamma}X_\gamma^\theta$.

Let $Y:=\bigoplus_{\alpha<\gamma}X_\alpha^\theta$ be the topological sum of the regular spaces $X_\alpha^\theta$ for $\alpha<\gamma$. It is clear that the space $Y$ is regular and the identity map $i:Y\to X$ is continuous. We claim that the identity map $i^{-1}:X\to Y$ is $\theta$-weakly discontinuous. Given any non-empty subset $A\subset X$ find the smallest ordinal $\beta\le\gamma$ such that $A\not\subset X_\beta$. Then $A\subset X_\alpha$ for all $\alpha<\beta$, which implies that $\beta$ is a successor ordinal. Write $\beta=\alpha+1$ for some $\alpha$ and observe that $U=A\cap X_\alpha^\theta=A\cap(X_\alpha\setminus X_{\alpha+1})$ is a non-empty $\theta$-open subspace of $A$ such that $i^{-1}|U$ is continuous. This means that $i^{-1}$ is $\theta$-weakly discontinuous and $i:X\to Y$ is a $\theta$-weak homeomorphism of $X$ onto the regular space $Y$.
\end{proof}

By analogy we can prove a characterization of weakly regular spaces.

\begin{engtheorem}\label{t:wr} 
A topological space $X$ is weakly regular if and only if each (closed) subspace $A\subset X$ contains a non-empty open regular subspace.
\end{engtheorem}

A topological space $X$ is called
\begin{itemize}
\item {\em quasi-regular} if each non-empty open subset of $X$ contains the closure of some non-empty open set in $X$;
\item {\em hereditarily quesi-regular} if each subspace of $X$ is quesi-regular.
\end{itemize}
Theorem~\ref{t:wt-char} implies

\begin{engcorollary}\label{c:wt=>hqr}   
Each $w\theta$-regular space is hereditarily quasi-regular.
\end{engcorollary}

Theorems~\ref{t:twr} and \ref{t:wr=>sw+twr} imply:

\begin{engcorollary}\label{cor1} 
Each scattered $T_1$-space is $\theta$-weakly regular and hence is $sw$-regular and $w\theta$-regular.
\end{engcorollary}

The $T_1$-requirement in Corollary~\ref{cor1} is essential as shown by the following example.

\begin{engexample}\label{ex} Consider the connected doubleton $D=\{0,1\}$ endowed with the topology $\big\{\varnothing,\{0\},\{0,1\}\big\}$. It is clear that $D$ is a scattered space. The function
$f\colon \mathbb R\to D$ defined by
$$f(x)=\begin{cases}
1&\mbox{if $x\in\mathbb Q$};\\
0&\mbox{if $x\notin\mathbb Q$}
\end{cases}
$$is scatteredly continuous but not weakly discontinuous as $C(f)=\mathbb Q$ has empty interior in $\mathbb R$. Consequently, $D$ is not $sw$-regular and hence not $\theta$-weakly regular.

The identity map $i:D\to \{0,1\}$ to the discrete doubleton is weakly discontinuous but not $\theta$-weakly discontinuous. This means that $D$ is not $w\theta$-regular.
\end{engexample}

\begin{engdefinition} 
A topological space $X$ is {\em locally regular} if $X$ admits an open cover by regular subspaces.
\end{engdefinition}

Theorem~\ref{t:wr} implies that each locally regular space is weakly regular.

\begin{engtheorem} 
Each locally regular topological space $Y$ is $sw$-regular.
\end{engtheorem}

\begin{proof}[\textsl{Proof}] 
Given a scatteredly continuous map $f:X\to Y$ and a non-empty subset $A\subset X$, we should show that the set $C(f|A)$ has non-empty interior in $A$.

By the scattered continuity of $f$, the map $f|A$ has a continuity point $a\in A$. By our assumption, the point $f(a)$ is contained in an open regular subspace $U\subset Y$. By the continuity of $f$ at $a$, there exists an open neighborhood $O_a\subset A$ of $a$ such that $f(O_a)\subset U$. Since $U$ is regular, the set $C(f|O_a)$ has non-empty interior in $O_a$ and then the set $C(f)\supset C(f|O_a)$ has non-empty interior in $A$.
\end{proof}

\begin{engexample}\label{ex:local} 
On the real line $\IR$ consider the Euclidean topology $\tau_E$ and the topology $\tau$ generated by the subbase
$$\tau_E\cup\{W_n:n\in\w\}\mbox{ \ where \ }W_n=\IR\setminus\left\{\tfrac1{2^k3^m}:m\in\w,\;k\ge n\right\}.$$
It can be shown that the space $X=(\IR,\tau)$ is $\theta$-weakly regular but not locally regular.
\end{engexample}

A topological space $X$ is called {\em regular at a point} $x\in X$ if any neighborhoodof $x$ in $X$ contains a closed neighborhood of $x$ in $X$.
A topological space $X$ is called {\em nowhere regular} if $X$ is not regular at each point $x\in X$. 

\begin{engexample}\label{ex:R} 
Let $\tau_E$ be the Euclidean topology of the real line and $\tau$ be the topology generated by the subbase $$\{(U\cap\IQ)\cup\{x\}:x\in U\in\tau_E\}.$$The space $(\IR,\tau)$ is locally regular and hence $sw$-regular. On the other hand, it is nowhere regular, not quasi-regular and not $w\theta$-regular.
\end{engexample}

Now, we describe the smallest non-regular first-countable Hausdorff space, which is called the Gutik hedgehog. The {\em Gutik hedgehog} is the space $\IN^{\le 2}=\IN^0\cup\IN^1\cup\IN^2$ endowed with the topology generated by the base 
$$\{\{x\}:x\in \IN^2\}\cup\{U_n:n\in\IN\}\cup\{U_{n,m}:n,m\in\IN\}$$where
$$U_n=\{\emptyset\}\cup\{(i,j)\in\IN^2:i\ge n\}\mbox{ and }U_{n,m}=\{(n)\}\cup\{(n,j):j\ge m\}\subset \IN^1\cup\IN^2$$for $n,m\in\w$. Here $\emptyset$ is the unique element of the set $\IN^0$. For the first time, the Gutik hedgehog has appeared in the paper \cite{GP} of Gutik and Pavlyk.


The following properties of the Gutik hedgehog can be derived from its definition.

\begin{englemma} The Gutik hedgehog is first-countable, scattered and locally regular, but not regular.
\end{englemma}

Moreover, the following theorem shows that the Gutik hedgehog is the smallest space among non-regular first-countable spaces.

\begin{engtheorem} A first-countable Hausdorff space $X$ is not regular if and only if $X$ contains a topological copy of the Gutik hedgehog.
\end{engtheorem}

\begin{proof} The ``if'' part follows from the non-regularity of the Gutik hedgehog.
\smallskip

To prove the ``only if'' part, assume that a first-countable Hausdorff space $X$ is not regular at some point $x$. Then we can find a neighborhood $U_0\subset X$ of $x$ that does not contain the closure of any neighborhood $V$ of $x$. Fix a neighborhood base $\{U_n\}_{n\in\IN}$ at $x$ such that $U_n\subset U_{n-1}$ for all $n\in\IN$. Let $k_1=0$, choose any point $x_1\in \overline{U}_{k_1}\setminus U_0$,  and using the Hausdorff property of $X$, find a neighborhood $V_1$ of $x_1$ such that $V_1\cap U_{k_2}=\emptyset$ for some number $k_2>k_1$. 

Proceeding by induction, we can choose an increasing number  sequence $(k_n)_{n\in\w}$ and a sequence $(x_n)_{n\in\IN}$ of points in $X$ such that for every $n\in\IN$, the point $x_n$ belongs to $\overline{U}_{k_n}\setminus U_0$ and has an open neighborhood $V_n$, disjoint with the neighborhood $U_{k_{n+1}}$ of $x$. Observe that for every $i<n$, we have $$x_n\in \overline{U}_{k_n}\subset\overline{U}_{k_i}\subset X\setminus V_i\subset X\setminus\{x_i\},$$ which implies that $x_n\notin\{x_i\}_{i<n}$. Replacing $V_n$ by a smaller neighborhood of $x_n$, we can assume that its closure $\overline{V}_n$ does not contain the points $x_1,\dots,x_{n-1}$.

Since $X$ is first-countable, for every $n\in\IN$ we can choose a sequence $\{x_{n,i}\}_{i\in\IN}$ of pairwise distinct points in $V_n\cap U_{k_n}$ that converge to $x_n$. 
Observe that for any $n<m$ the sets $\overline{U}_{k_{n+1}}\supset\overline{U}_{k_m}\supset\{x_{m,i}\}_{i\in\IN}$ and $V_n\supset \{x_{n,i}\}_{i\in\IN}$ are disjoint, which implies that the points $x_{n,i}$, $n,i\in\IN$, are pairwise disjoint.  Consider the subspace $\tilde H:=\{x\}\cup\{x_n:n\in\IN\}\cup\{x_{n,i}:n,i\in\IN\}$ and observe that the map $h:H\to \tilde H$, defined by $h(\emptyset)=x$, $h(n)=x_n$ and $h(n,m)=x_{n,m}$ for $n,m\in\IN$, is a homeomorphism.
\end{proof}

Finally let us draw a diagram of all provable implications between various regularity properties.
$$\xymatrix{
&\mbox{regular}\ar@{=>}[d]&\\
\mbox{$sw$-regular}&\mbox{$\theta$-weakly regular}\ar@{=>}[l]\ar@{=>}[d]\ar@{=>}[r]&\mbox{$w\theta$-regular}\ar@{=>}[d]\\
\mbox{locally regular}\ar@{=>}[r]\ar@{=>}[u]&\mbox{weakly regular}&\mbox{hereditarily}\atop\mbox{quasi-regular}
}
$$
\smallskip

Examples~\ref{ex:Bokalo},  \ref{ex:local} and \ref{ex:R} show that none of the implications $$
\begin{gathered}
\mbox{weakly regular $\Ra$ $sw$-regular},\\
\mbox{$\theta$-weakly regular $\Ra$ locally regular},\\
\mbox{locally regular $\Ra$ $w\theta$-regular}
\end{gathered}$$
holds in general.

\begin{engproblem} Is each $sw$-regular space weakly regular? quasi-regular?
\end{engproblem}

\begin{engproblem} Which properties in the diagram are preserved by products?
\end{engproblem}

\noindent{\bf Acknowledgements.} The authors express their sincere thanks to Alex Ravsky for careful reading the paper and many valuable suggestions improving the presentation.

\end{document}